\documentclass[a4paper]{amsart}
\usepackage{amsfonts,amsmath,amssymb,amsthm,cmtiup} 
\usepackage{mathrsfs}
\usepackage{cite}

\newtheorem{theorem}{Theorem}
\newtheorem*{theorem*}{Theorem}
\newtheorem{lemma}{Lemma} 
\newtheorem*{lemma*}{Lemma} 
 
\newtheorem*{corollary*}{Corollary}

\theoremstyle{definition} 
\newtheorem{remark}{Remark}
\newtheorem*{remark*}{Remark}

\newtheorem*{definition*}{Definition}

\newtheorem*{example*}{Example}

\newtheorem*{problem*}{Problem}

\newcommand{\ind}{\operatorname{ind}}

\def\cl#1{\skew3\overline{#1}}
\def\clx#1#2{\skew3\overline{#1}^{\raise.5pt\hbox{\scriptsize$\,#2$\!}}}

\def\clrx#1#2{\skew3\overline{#1\hspace{-0.1em}}^{\raise.5pt\hbox{\,\scriptsize$#2$}}}

\def\clax#1#2#3{\skew6\overline{#1_{\rlap{$\scriptstyle 
#2$}}}^{\raise.5pt\hbox{\hspace{0.2667em}\scriptsize$#3$}}\!
}
\def\clarx#1#2#3{\skew7\overline{\,#1_{\rlap{$\scriptstyle 
#2$}}}^{\raise.5pt\hbox{\hspace{0.2667em}\scriptsize$#3$}}\!
}
\def\Cl#1^#2{\cl{#1\kern-.09em}\kern.09em\vphantom{#1}^{#2}}

\newcommand\cA{\mathscr A}

\newcommand\e{\varepsilon}
\def\a{\alpha}
\newcommand\bR{\mathbb R}
\newcommand\bN{\mathbb N}

\let\le\leqslant
\let\ge\geqslant

\begin{document}

\title{No Product Theorem\\ for the Covering Dimension\\ of Topological Groups} 
\author{Ol'ga Sipacheva}

\begin{abstract}
Two (strongly) zero-dimensional Lindel\"of topological groups whose product has positive covering dimension are 
constructed. An example of a Lindel\"of (strongly) zero-dimensional space whose free and free Abelian topological 
groups are not strongly zero-dimensional is given.
\end{abstract}



\address{Department of General Topology and Geometry, Faculty of Mechanics and  Mathematics, 
M.~V.~Lomonosov Moscow State University, Leninskie Gory 1, Moscow, 199991 Russia}

\email{o-sipa@yandex.ru, osipa@gmail.com}

\maketitle

This paper is concerned with the covering dimension of topological groups. There are two definitions of covering 
dimension, in the sense of \v Cech and in the sense of Kat\v etov; following \cite{Ch}, we denote the former by 
$\dim$ and the latter by $\dim_0$. Recall that, given a topological space $X$, $\dim X$  ($\dim_0 X$) is the 
least integer $n\ge -1$ such that any finite open (cozero) cover of $X$ has a finite open (cozero) refinement of 
order $n$, provided that such an integer exists. If it does not exist, then $\dim X=\infty$ ($\dim_0 X = 
\infty$). For normal spaces, these dimensions coincide (see~\cite{Ch}). A space $X$ for which $\dim_0(X)=0$ is 
said to be \emph{strongly zero-dimensional}.

In \cite{88} Shakhmatov asked whether the inequality $\dim_0(G \times H) \le \dim_0 G + 
\dim_0 H$ holds for arbitrary topological groups $G$ and $H$. Various versions of this question can be found 
in \cite{A-vM}. In the paper, we construct two Lindel\"of topological groups $G$ and $H$ for which 
$\dim_0(G\times H)>\dim_0 G+\dim_0 H = 0$ (and $\dim(G\times H)>\dim G+\dim H = 0$), thereby answering (in the 
negative) Shakhmatov's question and Questions~6.9 and~6.14 in \cite{A-vM}. A modification of this example also 
gives a negative answer to Arkhangel'skii's 1981 question of whether the free (free Abelian) topological group of 
any strongly zero-dimensinal space is strongly zero-dimensional \cite{10} (see also \cite[Problem~8.17]{A-vM}).

\smallskip

Our construction of topological groups $G$ and $H$ for which $\dim_0(G \times H)>\dim_0 G+\dim_0 H$ is based on 
Przymusi\'nski's notion of $n$-cardinality and $n$-Bernstein sets \cite{Prz1} and on his construction 
of Lindel\"of spaces $X$ and $Y$ such that $X\times Y$ is normal and $\dim X=\dim Y= 0$ but $\dim (X\times Y) >0$ 
\cite{Prz2}. Below we recall some details, following the exposition of the construction given in~\cite{Ch}.

Let $X$ be a set, and let $n\in \mathbb N$. The \emph{$n$-cardinality} (with respect to $X$) of a set $A\subset 
X$, denoted by $|A|_n$, is the least cardinal $\kappa$ such that
$$
A\subset \bigcup_{i=1}^n (X^{i-1}\times Y\times X^{n-i})
$$
for some $Y \subset X$ with $|Y|=\kappa$ (here and in what follows it is assumed that $X^0\times Y=Y\times 
X^0=Y$). Clearly, $|A|_1= |A|$ and $|A|_n \le |A|$. If $|A|_n \le \omega$, then $A$ is 
said to be \emph{$n$-countable}; otherwise, $A$ is said to be \emph{$n$-uncountable}. 

Given $n\in \bN$, we say that a set $B 
\subset X$ is \emph{weakly $n$-Bernstein with respect to a topology} $\tau$ on $X^n$ if $|A\cap B^n|_n=2^\omega$ 
for every $n$-uncountable $\tau$-closed set $A\subset X^n$. 

For $x\in X^n$ and $i\le n$, by $x_i$ we denote the $i$th coordinate of $x$ and by $\tilde x$, the set of 
all coordinates of $x$: $\tilde x= \{x_1, \dots, x_n\}$.  

By abuse of notation, given a topology $\tau$ on $X$, we denote the product topology on $X^n$ by $\tau^n$.

\begin{lemma}[{see \cite[Lemma 24.1]{Ch}}]
\label{l1}
For a set $A \subset X^n$ and an infinite cardinal $\kappa$, the following
conditions are equivalent:
\begin{enumerate}
\item[(a)]
$|A|_n = \kappa$; 
\item[(b)]
$A$ contains a subset $B$ of cardinality $\kappa$ such that $\tilde p\cap \tilde q =\varnothing$ whenever $p$ and 
$q$ are distinct points of $B$.
\end{enumerate}
\end{lemma}

The following theorem is based on Theorem~24.3 and Proposition~24.4 in~\cite{Ch}.

\begin{theorem}
\label{th1}
Let $(X,\tau)$ be a space with separable completely metrizable topology $\tau$, and let $\mu$ be a
topology on $X^2$ with the following properties: 
\begin{enumerate}
\item[(i)]
$\mu\supset \tau^2$;
\item[(ii)]
$X^2$ contains at most $2^\omega$ \,$2$-uncountable $\mu$-closed sets;
\item[(iii)]
$|A|_2\ge 2^\omega$ for any $2$-uncountable $\mu$-closed set $A\subset X^2$.
\end{enumerate}
Then $X$ contains pairwise disjoint sets $B_1,B_2, \dots$ such that every $B_i$ is weakly $2$-Bernstein with 
respect to $\mu$ and weakly $n$-Bernstein with respect to $\tau^n$ for all $n$.  
\end{theorem}  

\begin{proof}
First, note that, for each $n\in \bN$, the number of $n$-uncountable $\tau^n$-closed sets in $X^n$ does not 
exceed $2^\omega$ and the $n$-cardinality of each of them is at least $2^\omega$ (see Theorem~24.2 and the 
beginning of the proof of Theorem~24.3 in \cite{Ch}). By $\cA_2$ we denote the family of all 
$2$-uncountable $\mu$-closed subsets of $X^2$ and by $\cA_n$,  $n\ne 2$,  the family of all $n$-uncountable 
$\tau^n$-closed subsets of $X^n$. Note that $\cA_2$ contains all $2$-uncountable 
$\tau^2$-closed subsets of $X^2$, because $\tau^2\subset \mu$. We set $\cA=\bigcup_{n\in \bN}\cA_n$ and 
enumerate all elements of $\cA$ by ordinals in $2^\omega$ as $\cA=\{A_\alpha: \a\in 2^\omega\}$ so that each 
element occurs $2^\omega$ times in this enumeration. For each $\a <2^\omega$, we denote the unique $n\in \bN$ for 
which $A_\a\in \cA_n$ by $n(\a)$ and proceed precisely as in the proof of Theorem~24.3 of \cite{Ch}. Namely, 
using transfinite induction, we choose points $p(\a,i)\in \cA_\a$ for all $\a\in 2^\omega$ and $i\in \bN$ so 
that $\tilde p(\a,i)\cap \tilde p(\beta,j)=\varnothing $ if $\a\ne \beta$ or $i\ne j$. Let $p(0,i)$, $i\in 
\bN$, be arbitrary points in $A_0$ satisfying the condition $\tilde p(0,i)\cap \tilde p(0,j)=\varnothing $ for 
$i\ne j$ (they exist by Lemma~\ref{l1}). Suppose that $\gamma>0$ and points $p(\a,i)$ are already chosen 
for all $\a<\gamma$ and $i\in \bN$. Note that the cardinality of the set 
$$
Y=\bigcup \{\tilde p(\a,i): \a<\gamma, i\in \bN\}
$$
is less than $2^\omega$. By assumption and in view of the remark at the beginning of the proof we have 
$|A_\gamma|_{n(\gamma)}=2^\omega$. By Lemma~\ref{l1} there exists a $B\subset A_\gamma$ such that $|B|=2^\omega$ 
and $\tilde p\cap \tilde q =\varnothing $ for any distinct $p,q\in B$. As $p(\gamma,1),  p(\gamma, 
2), \dots$ we choose any different points in 
$$
B\setminus \bigcup_{i=n}^{n(\gamma)} (X^{i-1}\times Y\times X^{n(\gamma)-i}).
$$

We set 
$$
B_i=\bigcup \{\tilde p(\a,i):\a<2^\omega\}
$$ 
for each $i\in \bN$. Clearly, $B_i\cap B_j=\varnothing $ if $i\ne j$. Any $n$-uncountable 
$\tau^n$-closed subset $A$ of $X^n$ equals $A_\alpha$ for $2^\omega$ indices $\a\in 2^\omega$, and we 
have $p(\a,i)\in A\cap B^{n(\a)}_i$ and $n(\a)=n$ for each of these $\a$ and all $i\in \bN$. Since $\tilde p(\a, 
i)\cap \tilde p(\beta, i)$ for $\a\ne \beta$, it follows that $|A\cap B_i^n|_n\ge 2^\omega$ by Lemma~\ref{l1}. 
Similarly, we have $|A\cap B_i^2|_2\ge 2^\omega$ for any $2$-uncountable 
$\mu$-closed subset $A$ of $X^2$.
\end{proof}

Let $C$ be the Cantor set in $[0,1]\subset \bR$, and let $\tau$ be the usual (Euclidean) topology on $C$. In 
\cite[proof of Theorem~27.5]{Ch} (with a reference to \cite{Prz2}) a topology $\mu$ on $C^2$ satisfying 
conditions (i)--(iii) in Theorem~\ref{th1} for $X=C$ was defined and, given an arbitrary decomposition $\{S, S_1, 
S_2\}$ of $C$ into pairwise disjoint sets weakly $2$-Bernstein with respect to $\mu$, topologies 
$\tau_1$ and $\tau_2$ on $C$ were constructed, which have, in particular, the following properties: for $i=1,2$, 
\begin{enumerate}
\item
$\tau_i\supset \tau$; 
\item
any $\tau_i$-neighborhood of any point of $S_i$ is a $\tau$-neighborhood; 
\item
$\tau_i$ has a base consisting of $\tau$-closed sets;
\item
$\dim(C, \tau_i)=\dim_0(C, \tau_i)=0$; 
\item 
$(C,\tau_1)\times (C,\tau_2)$ is normal and $\dim((C,\tau_1)\times (C,\tau_2))=\dim_0((C,\tau_1)\times 
(C,\tau_2))=1$. 
\end{enumerate} 

Using Theorem~\ref{th1}, we can choose $S_1$ and $S_2$ in the above construction to be $n$-weakly Bernstein with 
respect to $\tau^n$ for all $n\in \bN$. We fix the corresponding topologies $\tau_1$ and $\tau_2$ and set  
$C_i=(C, \tau_i)$ for $i=1,2$.  

\begin{lemma}
\label{l2}
The spaces $C_i^n$ are Lindel\"of for all $n\in \bN$. 
\end{lemma}

\begin{proof}
We argue by induction on $n$.

Let $\gamma$ be a $\tau_i$-open cover of $C_i$. In view of (2), each point $s\in S_i$ has a $\tau$-open 
neighborhood $U_s$ contained in an element $V_s$ of $\gamma$. Let $U=\bigcup_{s\in S_i} U_s$. Since $S_i$ is 
weakly $1$-Bernstein with respect to $\tau$ and $C\setminus U$ is a $\tau$-closed set disjoint from $S_i$, it 
follows that $C\setminus U$ is $1$-countable, that is, countable. For each $x\in C\setminus U$, choose an 
element $V_x$ of $\gamma$ containing $x$. Let $\{U_{s_k}:k\in \bN\}$ be a countable subcover of the $\tau$-open 
cover $\{U_s:s\in S_i\}$ of $S_i$. Then $\{V_{s_k}:k\in \bN\}\cup \{V_x:x\in C\setminus U\}$ is a countable 
subcover of $\gamma$. 

Suppose that $n>1$ and we have already proved that $C_i^k$ is Lindel\"of for all $k<n$. Let $\gamma$ be a 
$\tau_i^n$-open cover of $C_i^n$. In view of (2), each point $s\in S_i^n$ has a $\tau^n$-open neighborhood $U_s$ 
contained in an element $V_s$ of $\gamma$. Let $U=\bigcup_{s\in S_i^n} U_s$. Since $S_i$ is 
weakly $n$-Bernstein with respect to $\tau^n$ and $C^n\setminus U$ is a $\tau$-closed set disjoint from $S_i^n$, 
it follows that $C^n\setminus U$ is $n$-countable, that is, there exists a countable set $Y\subset C$ such that 
$C^n\setminus U$ is contained in a countable union of spaces of the form $C_i^{k-1}\times \{x\}\times C_i^{n-k}$, 
each of which is Lindel\"of by the induction hypothesis. It remains to choose a countable subfamily of $\gamma$ 
covering $C^n\setminus U$ and a countable subfamily of $\{V_s:s\in S_i^n\}$ covering $U$, which exists because  
$\{V_s:s\in S_i^n\}$ has the $\tau^n$-open refinement $\{U_s:s\in S_i^n\}$. 
\end{proof}

For a completely regular Hausdorff space $X$, let $F(X)$ and $A(X)$ denote, respectively, its free and free 
Abelian topological groups (see, e.g., \cite{AT}).

\begin{theorem}
\label{th2}
For the spaces $C_1$ and $C_2$ defined above, 
$\dim A(C_1)=\dim_0 A(C_1)=\dim A(C_2)=\dim_0 A(C_2)= 0$, while $\dim (A(C_1)\times A(C_2))>0$ and $\dim_0 
(A(C_1)\times A(C_2))>0$. 
\end{theorem}

\begin{proof}
It is well known (see, e.g., \cite[p.~417]{AT}) that, for any $X$, the group $A(X)$ is the union of its closed 
subspaces $A_n(X)=\{\e_1 x_1+ \dots + \e_n x_n : \e_i=\pm1, x_i\in X\}$, each of which is the continuous image 
under the natural addition map of the $n$th power of the disjoint union $X\oplus \{0\}\oplus X^{-1}$, where 
$X^{-1}$ is a homeomorphic copy of $X$ and $\{0\}$ is a singleton ($0$ is the zero element of $A(X)$). Thus, 
$A(C_i)$ is Lindel\"of for $i=1,2$. Since $\dim C_i=0$, we have $\ind A(C_i)=0$ \cite{Tk-0} and hence $\dim 
A(C_i)=\dim_0 A(C_i)=0$, because the covering dimension of a Lindel\"of space does not exceed its small 
inductive dimension and $\dim=\dim_0$ for normal spaces (see~\cite{Ch}). 

According to \cite[Theorem~7 (version~2)]{3}, $C_i$ is a retract of $A(C_i)$. Hence $C_1\times C_2$ is a 
retract of $A(C_1)\times A(C_2)$. Thus, $C_1\times C_2$ is $C$-embedded in $A(C_1)\times A(C_2)$, which, 
together with (5) and Theorem~11.22 of \cite{Ch}, implies $\dim_0 (A(C_1)\times A(C_2))>0$. Clearly, any space 
$X$ with $\dim X=0$ is strongly zero-dimensional (because any disjoint open cover is cozero), whence $\dim 
(A(C_1)\times A(C_2))>0$. 
\end{proof}

\begin{theorem}
\label{th3}
The space $X=C_1\oplus C_2$ has the following properties:
\begin{enumerate}
\item[(i)] 
$\dim X=\dim_0 X=0$;
\item[(ii)]
$\dim A(X) >0$ and $\dim_0 A(X)>0$;
\item[(iii)]
$\dim F(X)>0$ and $\dim_0 F(X)>0$.
\end{enumerate}
\end{theorem}

\begin{proof}
Property (i) obviously follows from property (4) of the topologies $\tau_i$. Property (ii) follows from 
Theorem~\ref{th2} and the fact that the group $A(C_1)\times A(C_2)$ is topologically isomorphic to $A(C_1\oplus 
C_2)$~\cite[Proposition~4]{Tkachuk}. The  isomorphism $i\colon A(C_1\oplus C_2)\to A(C_1)\times A(C_2)$ takes 
each point $x\in C_1\oplus C_2$ to $(x,0_2)$ if $x\in C_1$ and to $(0_1, x)$ if $x\in C_2$ (by $0_i$ we denote 
the zero element of $A(C_i)$). 

Let us prove (iii). 
The space $X^2$ is topologically embedded in $F(X)$ as a closed subspace consisting of two-letter words of 
the form $xy$, where $x,y\in X$ (see \cite[Theorem~7.1.13]{AT}). Therefore, $C_1\times C_2$ is topologically 
embedded in $F(X)$ as a closed subspace $Y$ consisting of two-letter words of the form $xy$, where $x\in C_1$ and 
$y\in C_2$. Let us show that $Y$ is $C$-embedded in $F(X)$.  

Recall that $A(X)$ is a topological quotient of $F(X)$. Let $h\colon F(X)\to A(X)$ be the natural quotient 
homomorphism. It takes $Y$ to the set $h(Y)$ of all elements of $A(X)$ of the form $x+y$, where $x\in C_1$ and 
$y\in C_2$. Note that $i(x+y)=(x,y)\in A(C_1)\times A(C_2)$ for any such $x$ and~$y$, so that the restriction 
$h|_Y\colon Y\to h(Y)$ is one-to-one. Moreover, $h|_Y$ is a homeomorphism, because its inverse is the composition 
of the homeomorphism $i|_{h(Y)}\colon h(Y)\to C_1\times C_2\subset A(C_1)\times A(C_2)$ and the natural 
multiplication map $C_1\times C_2\to F(C_1\oplus C_2)$. 

As mentioned above, $C_i$ is a retract of $A(C_i)$ for 
$i=1,2$, and hence $C_1\times C_2$ is a retract of $A(C_1)\times A(C_2)$. Let $r\colon A(C_1)\times A(C_2) \to 
C_1\times C_2$ be a retraction. Then $r\circ i\colon A(C_1\oplus C_2)\to C_1\times C_2$ is a continuous map whose 
restriction to $h(Y)$ is a homeomorphism $h(Y)\to C_1\times C_2$. 

Take any continuous function $f\colon Y\to \bR$. Since $C_1\times C_2$ is a retract of $A(C_1)\times A(C_2)$,  
it follows that $i(h(Y))=C_1\times C_2$ is $C$-embedded in $A(C_1)\times A(C_2)$. Let $\hat f$ be a continuous 
extension of $f\circ \bigl(h|_Y\bigr)^{-1}\circ i^{-1}|_{C_1\times C_2}\colon C_1\times C_2\to \bR$ 
to $A(C_1)\times A(C_2)$. Then $\hat f\circ i\circ h\colon F(X)\to \bR$ is a continuous extension of $f$ to 
$F(X)$. 

Thus, $Y$ is $C$-embedded in $F(X)$. Since $Y$ is homeomorphic to $C_1\times C_2$, we have $\dim_0 Y>0$ (see 
property (5) of the topologies $\tau_i$). Therefore, $\dim_0 F(X)>0$ by Theorem~11.22 of \cite{Ch}, and hence 
$\dim F(X)>0$. 
\end{proof}

\begin{remark}
the group $A(C_2)$ in our example can be made to have countable network weight. Indeed, 
setting $C_2=(S_2, \tau_2|_{S_2})$ instead of $C_2=(C,\tau_2)$, we obtain an example of two strongly 
zero-dimensional spaces $C_1$ and $C_2$, one Lindel\"of to any finite power and the other 
separable and metrizable, for which $\dim_0(C_1\times C_2)>0$~\cite{Prz3, Ch}. 
\end{remark}

\begin{remark}
The space $X$ in Theorem~\ref{th3} is Lindel\"of.
\end{remark}

\end{document}